\documentclass[12pt,reqno]{amsart}
\usepackage{amssymb}
\usepackage{latexsym}

\usepackage{amsmath,amsfonts,amsthm,amsopn,cite,mathrsfs}
\usepackage{epsfig,verbatim}
\usepackage{subfigure}

\numberwithin{equation}{section}

\newcommand{\ZZ}{\mathbb Z}

\newcommand{\beq}{\begin{equation}}
\newcommand{\eeq}{\end{equation}}

\newcommand{\area}{\mathop{\rm Area\,}}
\newcommand{\card}{\mathop{\rm Card\,}}
\newcommand{\diam}{\mathop{\rm diam\,}}
\newcommand{\dist}{\mathop{\rm dist}}

  \def\cB{\mathcal{B}}  
    
\def\cF{\mathcal{F}}
\def\cG{\mathcal{G}}

\def\cN{\mathcal{N}}
\def\cZ{\mathcal{Z}}

 \newcommand{\lng}{\langle}
 \newcommand{\rng}{\rangle}

\newcommand{\vf}{{g_0}}
 \def\bC{\mathbb{C}}
  
 \def\bR{\mathbb{R}}
 
   \def\bZ{\mathbb{Z} }

  \def\cB{\mathcal{B}}  
   
\def\cF{\mathcal{F}}
\def\cG{\mathcal{G}}

\newcommand{\inv}{^{-1}}

\def\lr{L^2(\bR)}
\newcommand{\tfs}{time-frequency shift}

\newcommand{\fif}{if and only if}

\newtheorem{th1}{{\bf Theorem}}[section]

\newtheorem{lem}[th1]{{\bf Lemma}}

\newtheorem{prop}[th1]{{\bf Proposition}}
\newtheorem{cor}[th1]{{\bf Corollary}}

\theoremstyle{definition}
\newtheorem{rem}[th1]{{\bf Remark}}

\begin{document}
\author {A. Borichev}
\address{Centre de Math\'ematiques et Informatique, Universit\'e d'Aix-Marseille I, 39, rue Joliot Curie, 13453, Marseille Cedex
13, France}
 \email{borichev@cmi.univ-mrs.fr}

\author{K. Gr\"ochenig}
\address{Faculty of Mathematics \\
University of Vienna \\
Nordbergstrasse 15 \\
A-1090 Vienna, Austria}
\email{karlheinz.groechenig@univie.ac.at}

 \author{Yu. Lyubarskii}
\address{Department of Mathematical Sciences, Norwegian University of Science and Technology, NO-7491, Trondheim, Norway}
\email{yura@math.ntnu.no}

\title[ Frame constants  near the  critical density]
{  Frame Constants of Gabor Frames near the Critical Density}

\thanks{A.B. was partially supported by the ANR projects 
ANR-07-BLAN-0249 and ANR-09-BLAN-0058-01;
K.G. was  supported by the  Marie-Curie Excellence Grant
  MEXT-CT-2004-517154; Yu.L. was partly supported by the Research
Council of Norway, grants 160192/V30 and  177355/V30.}

\begin{abstract}
 We consider  Gabor frames  generated by a Gaussian function    and
 describe the behavior of  the frame constants as the  density of the
 lattice approaches the critical value. 
 \end{abstract}


\maketitle
\section{Introduction}

In this article we study the stability problem for  the expansions of functions 
on the real line with respect to a discrete set of phase-space shifts of a Gaussian,
precisely
\begin{equation}
\label{eq:1} 
f(x) = \sum _{k,l\in \bZ } c_{kl} e^{2\pi i lax} e^{-\pi (x-bk)^2} \, .
\end{equation}
Expansions of such form (with $a=1$, $b=1$) were introduced by D.~Gabor in his
classical article  \cite{Gabor}. Now   expansions of type~\eqref{eq:1}, so-called Gabor
expansions, appear in signal processing, quantum mechanics,
time-frequency analysis, the theory of pseudodifferential operators, and other
applications. 

During the last decades an extensive theory of expansions \eqref{eq:1} as well as more general Gabor expansions has been developed (see, for instance \cite{Charlybook, Christensen} and the references therein). However, not much  is known about numerical  stability property of such expansions.

In modern language, the existence of Gabor expansions is derived from
frame theory. To fix  terminology and notation, take some   $g\in \lr$, it will be called  a
 window function,  and  let  $\Lambda =  M \bZ ^2  \subset \bR ^2$ be a  lattice
in $\bR ^2$, where $M $ is a $2\times 2$ invertible real-valued
matrix.
Given a point $\lambda=(x,\xi )$ in phase-space $\bR ^2$, the 
corresponding \tfs\ is 
$$
\pi_\lambda f(t) = e^{2\pi i \xi t } f(t-x),        \qquad  t\in \bR
\, .
$$
The set of functions 
   $\cG (g,\Lambda ) = \{\pi_\lambda g :  \lambda  \in \Lambda\}$ is
called the \emph{Gabor system} generated by $g$  and $\Lambda$.
We say that such a system is a {\em Gabor frame}
 or \emph{Weyl-Heisenberg frame}, whenever
there  exist  constants $A,B >0$ such that, 
for all $f\in \lr$,
\begin{equation}
\label{Eq:2.8} 
A\|f\|^2_{\lr} \leq \sum_{\lambda \in
\Lambda}|\langle f, \pi_\lambda g\rangle_{\lr}|^2 \leq
B\|f\|^2_{\lr}\, .
\end{equation}
The (best possible)  constants $A=A(\Lambda,g)$  and 
$B=B(\Lambda,g)$ in~\eqref{Eq:2.8}  are called the 
{\em lower and upper frame bounds}  for the  frame  $\cG (g,\Lambda )$.

There is a standard procedure for constructing expansions of type \eqref{eq:1} 
for each Gabor frame. Namely there exists a 
\emph{dual window} $\gamma \in L^2(\bR )$, such that every  $f\in
L^2(\bR )$ can be
expanded as a Gabor series 
\begin{equation}
\label{eq:2}
f= \sum _{\lambda \in \Lambda } \langle f, \pi_\lambda \gamma
\rangle  \pi_\lambda  g.
\end{equation}
Such dual window is, in general, non-unique. 
We refer the reader to e.g. ~\cite{Charlybook} for an exposition of Gabor analysis and related matters.

The property of  the system $\cG (g,\Lambda )$ to form a frame in $\lr$ depends (among other factors) on geometrical characteristics of $\Lambda$. We say that the area of its fundamental 
domain  $s(\Lambda ) = \mbox{Area}(M[0,1)^2)= |\det M|$  is the \emph{size}  of $\Lambda$. 
By the {\em density} of $\Lambda$  we mean $d(\Lambda)=s(\Lambda)^{-1}$; for the lattice case this definition of density coincides with numerous   
standard density definitions (see e.g. \cite{seip}).

We refer to~\cite{heil} for a comprehensive account of the density
theorems for Gabor frames. In fact  for \emph{any} window function $g$ 
 the condition $s(\Lambda )\leq 1$ is necessary 
for $\cG (g, \Lambda )$ to be a frame in $\lr$. 
   For ``nice'' windows $g$ (in the Schwartz class, say) a
fascinating form of the uncertainty principle, the so-called
Balian-Low Theorem (BLT), requires even that  $s(\Lambda )<1$ for $\cG
(g, \Lambda )$ to be a frame~\cite{BHW95}.  

The results of \cite{sw,lyu} yield in particular that in the case of 
 the
Gaussian window 
$$
\vf (t) = e^{-\pi t^2},\qquad  t\in \bR .
$$
the condition  $s(\Lambda ) <1$ is also sufficient:
\medskip

\noindent{\bf Theorem A.} {\em  
  The set $\cG (\vf , \Lambda )$ is a frame for $\lr $ \fif\
  $s(\Lambda )<1$ .}
 
\medskip

Together with BLT this implies that  the lower frame bound
$A=A(\Lambda )$ must tend to $0$, as the size of the lattice
$s(\Lambda ) $ approaches  one. Thus the original Gabor series  
 \eqref{eq:1} with $a=1$, $b=1$   corresponds to the critical case $s(\Lambda )=1$
 and  does not provide an $L^2$-stable  expansion. 
 \footnote{In his article \cite{Gabor} Gabor considered  expansions of functions $f$
 which possess additional decay in time and frequency. As we now know  (see e.g. \cite{LS}),  for such functions the series \eqref{eq:1} converges in $\lr$. }
 In this case, there exist  $L^2$-functions with  polynomially growing
coefficients. See~\cite{janssen, LS} for the convergence properties
of~\eqref{eq:1}.

In this article we are concerned exclusively with Gabor frames for the
Gaussian window $\cG (\vf , \Lambda )$  for the square lattice 
$\Lambda(a)=a\ZZ\times a\ZZ$ and study the behavior of its frame constants 
$A(a)=A(\Lambda(a))$  and $B(a)=B(\Lambda(a))$ near the critical density $d(\Lambda)=1$. The main result of the article reads as follows.

\begin{th1}
\label{th:1}
There exist constants $0<c<C<\infty$ such that for each $a\in (1/2,1)$ the frame bounds $A(a)$,
$B(a)$ for the frame $\cG (\vf , \Lambda(a) )$ satisfy
\begin{equation}
\label{eq:3}
c (1-a ^2)\leq A(a ) \leq C (1-a ^2)
\end{equation}
and
\begin{equation}
\label{xlst}
c< B(a) < C.    
\end{equation}
\end{th1}

\begin{rem}
A similar statement holds for arbitrary rectangular lattices.  
  The values of  $c, C$ in this theorem then  depend upon the shape of
  the lattice. Nevertheless,  one can prove
  that given a number $K>0$ there exist 
  constants  $c$ and $C$ valid 
  for all matrices $M$ such that  the diameter  
of the fundamental domain $M (0,1]^2$  does not exceed $K$.
\end{rem}

The ratio $B(\Lambda)/A(\Lambda)$ plays the role of the condition number for the frame 
$\cG (\vf , \Lambda)$. Thus Theorem \ref{th:1}  says how fast does the frame 
$\cG (\vf , \Lambda)$ "numerically degenerate"  as its density approaches the critical value.

The asymptotical behavior $A(a ) \asymp  (1-a ^2) $ has been first observed
numerically by Thomas Strohmer~\cite{str09}  and by Peter Sondergaard~\cite{So}.
Moreover, the numerical simulation in \cite{str09} allowed us to guess the construction which 
gives 
the second i
nequality in \eqref{eq:3}. This construction is described 
in Section 4 below.

Next let $g_1(t) = (\mathrm{cosh}\, \pi \gamma t)\inv$, $\gamma >0$,
be the hyperbolic cosine function. Janssen and Strohmer~\cite{JS02}
have shown that $\cG (g_1, a\bZ \times b\bZ)$ is a frame, if and only
if $ab<1$. To do this, they showed that the frame bounds for $\cG (g_1, a\bZ \times b\bZ)$
are equivalent to those of $\cG (g_0, a\bZ \times b\bZ)$ with the
Gaussian $g_0$ and applied Theorem~A.  Therefore we obtain the same
asymptotic estimates for the frame bounds for the hyperbolic cosine.

\begin{cor}
There exist constants $0<c<C<\infty$ such that for each $a\in (1/2,1)$
the frame bounds $\tilde{A}(a)$, 
$\tilde B(a)$ for the frame $\cG (g_1 , \Lambda(a) )$ satisfy
$$
c (1-a ^2)\le \tilde A(a ) \le C (1-a ^2)
$$
and
$$
c\le \tilde B(a) \le C.    
$$
  
\end{cor}

The proof of Theorem~\ref{th:1} involves both time-frequency
methods and methods of complex analysis. 
We use complex analysis in order to obtain the upper estimate for $A(a)$
and the Gabor analysis  in order to obtain  the rest of the statements in 
Theorem 1.1 (though a pure complex-analytic proof is also available).
 In particular we apply
  Walnut's estimates for the norm of the frame
operator~\cite{walnut92}, and also precise decay estimates for the
dual window established in \cite{superframes}.  The upper bound $A(a) \le C(1-a^2)$
will be established by the construction of a concrete example. 
We produce a function $f_a$ (depending on
the lattice $\Lambda(a)$), such that 
$$
\sum_{\lambda \in \Lambda(a) } |\langle f_a, \pi_\lambda \vf \rangle |^2
\le C(1-a^2) \|f_a\|_{\lr}^2 \, .
$$
By using the Bargmann transform, we   translate our
problem  into one of finding entire functions in
the Bargmann-Fock space whose restrictions to $\Lambda(a)$ 
are "small" with respect to 
their Fock norms.

The paper is organized as follows. In the next section we discuss the estimates for $B(a)$. 
Furthermore, we give the lower estimates for $A(a)$. Here we
mainly follow the arguments from \cite{superframes}.  
In section 3 we recall the definition of the Fock space $\cF$ of entire functions, and discuss the 
relations between the frame property of the system $\cG (\vf,\Lambda(a))$ and sampling in $\cF$. We also recall  basic properties of the Weierstrass $\sigma$-function. In section 4 we use these facts and also special "atomization"  techniques in order to construct the example which delivers the upper estimate in \eqref{eq:3}.

 \textbf{Notation:} To avoid dealing with too many intermediate
 constants, we use the standard notation $f \prec g$ to express
 an  inequality $f(x) \leq C g(x)$ for all $x$ with a constant $C$ 
 independent of $x$ (and possibly other parameters). Likewise, $f
 \asymp g$ means that there exist $A,B>0$ such that $A f(x) \leq g(x)
 \leq Bf(x)$ for all $x$.


\section{Time-Frequency Methods To Estimate Frame Bounds}

The estimate \eqref{xlst} on the upper frame bound $B(a)$ can be obtained in various ways.
In particular, we can use Walnut's estimates, which give a sufficient condition for
the Gabor frame operator to be bounded~\cite{walnut92}. This result also follows from the Polya-Plancherel type
inequalities for functions in the Bargmann-Fock space, see below Section 3 for more details.

To obtain the lower estimates for $A(a)$ we need to show the invertibility of the Gabor frame operator
and to estimate the norm of the inverse operator. We will approach this problem by using information
about a suitable dual window $\gamma $ and then apply Walnut's estimates to the Gabor expansion~\eqref{eq:2}.

To state Walnut's result we need the following definitions. 

Let $W$ be the the Wiener amalgam space of functions on the real line defined by the norm 
$$
\|g\|_W = \sum _{k\in \bZ } \sup _{t\in [0,1]} |g(t+k)| \, .
$$

Given a function $g$ in $L^2(\mathbb R)$, consider the Gabor system $\cG
(g,\Lambda )$ and the corresponding synthesis operator $D_{g,\Lambda }$,
$$
D_{g,\Lambda } \mathbf{c} = \sum _{\lambda \in \Lambda } c_\lambda \pi_\lambda  g,
$$ 
and the analysis operator $C_{g,\Lambda }$,
$$
(C_{g,\Lambda } f)(\lambda ) = \langle f, \pi_\lambda g\rangle,
\qquad \lambda \in \Lambda.
$$ 

If $D_{g,\Lambda }$ acts continuously from $\ell^2(\Lambda)$ to $L^2(\mathbb R)$, then
$C_{g,\Lambda }$ acts continuously from $L^2(\mathbb R)$ to $\ell^2(\Lambda)$,
and $C_{g,\Lambda } = D_{g,\Lambda }^*$.

The following lemma  from \cite{walnut92} gives an estimate for 
$\|D_{g,\Lambda}\|_{l^2\to L^2}$:

\begin{lem}\label{walnut}
  If $g \in W$ and $\Lambda = a\bZ^2$, then $D_{g,\Lambda
  }$  is bounded from $\ell ^2(\Lambda )$ to $L^2(\bR)$ and 
$$
    \|D_{g,\Lambda }\|_{L^2(\mathbb R)} \leq (1+a\inv ) \|g\|_W \, .
$$
\end{lem}

Since, obviously, in our situation $B(a)=\|C_{g,\Lambda}\|^2_{L^2\to l^2}=
\|D_{g,\Lambda}\|^2_{l^2\to L^2}$, we obtain

\begin{cor}
  If $g\in W$ and $a>0$, then 
\begin{equation}
\label{eq:c10}
B(a) \le (1+a\inv )^2 \|g\|_W^2 \, .
\end{equation}
\end{cor}

To treat Gabor frames with Gaussian window, we need to evaluate the amalgam space
norm of functions with Gaussian decay. 

\begin{lem}\label{gaussdecay}
Assume that $\kappa>0$, $|\gamma (t)| \leq  e^{-\pi \kappa t^2}$. Then 
\begin{equation}
\label{eq:c11}
\|\gamma \|_W \le  2+\kappa ^{-1/2}.
\end{equation}
\end{lem}

\begin{proof}
For $n\ge 1$, $n\in \bZ$, we have 
$$
\sup_{t\in [0,1]} |\gamma (n+t)| \le e^{-\pi \kappa n^2} \le \int_{n-1}^n e^{-\pi \kappa t^2}\, dt \, ,
$$
and likewise for $n<-1$, $n\in \bZ $, we have  
$$
\sup _{t\in [0,1]} |\gamma (n+t)| \leq e^{-\pi \kappa (|n|-1)^2} \leq \int
_{|n|-2}^{|n|-1} e^{-\pi \kappa t^2}\, dt \, .
$$ 
Consequently, 
$$
\|\gamma \|_W = \sum _{n\in \bZ } \sup _{t\in [0,1]} |\gamma(t+n)|
\leq 2+ \int _\bR e^{-\pi \kappa t^2}\, dt = 2+\kappa ^{-1/2} \, .
$$
\end{proof}

As a consequence we obtain an estimate on the upper frame bound of
Gaussian Gabor frames. 

\begin{prop} \label{upper}
The upper frame bound $B(a)$ of $\cG (\vf , a\bZ^2 )$, $1/2<a<1$, satisfies the estimate
$$
1<B(a)< 100.
$$
\end{prop}
\begin{proof}
  
For the upper estimate we use \eqref{eq:c10} and \eqref{eq:c11} with $\kappa =1$. 

To get the lower estimate we consider the sum  \eqref{Eq:2.8} for $f= g = \vf $. Then
$$
\sum _{\lambda \in \Lambda (a)} |\langle \vf, \pi_\lambda  \vf \rangle|^2 > \|\vf\|^2,
$$
which yields the desired estimate. 
\end{proof}

The time-frequency methods also yield the lower estimate in \eqref{eq:3}. This estimate requires 
the existence and some knowledge about a dual window. If $\cG (g, \Lambda )$ is a frame, then 
by the frame theory there exists a dual window $\gamma \in L^2(\bR )$, such that every $f\in L^2(\bR )$ 
possesses a(n unconditionally convergent) series expansion (Gabor expansion) of the form
$$
f= \sum _{\lambda \in \Lambda } \langle f, \pi_\lambda g\rangle \pi_
\lambda \gamma = D_{\gamma , \Lambda } C_{g,\Lambda }f \, .
$$
For the square lattice $\Lambda(a)$, 
Lemma~\ref{walnut} yields the following bound:
$$
  \|f\|_{\lr}^2 \leq (a\inv +1)^2 \|\gamma \|_W^2 \, \sum
  _{\lambda \in \Lambda } |\langle f, \pi_\lambda g\rangle |^2 \, .
$$
Consequently, the lower frame bound $A(a)$ can be estimated from below as
\begin{equation}
\label{eq:c17}
A(a) \ge \bigl((a\inv +1)^2  \|\gamma \|_W^2\bigl)\inv \, .  
\end{equation}
 
\begin{prop}\label{lower}
For the square lattice $\Lambda(a)$, $1/2<a<1$,  the lower
bound $A(a)$ of the Gaussian frame $\cG (\vf , \Lambda (a) )$  obeys
the   estimate 
$$
  A(a) \succ 1-a^2.
$$
\end{prop}

\begin{proof}
In \cite{superframes} the authors consider 
the Gaussian Gabor frame $\cG (\vf , \Lambda (a) )$. For this frame they 
construct a dual window $\gamma$ such that
$$
|\gamma (t) | \leq C e^{-\pi \kappa t^2}
$$
with $\kappa \asymp 1-a^2$.

By Lemma~\ref{gaussdecay} we have 
$$
\|\gamma \|_W \prec 2+\kappa ^{-1/2} \prec (1-a^2)^{-1/2},
$$
and the desired estimate follows now from \eqref{eq:c17}. 
\end{proof}

\section{Complex Methods}


\subsection{Fock space\label{FM}}
  
We recall the definition and basic properties of the Fock space.
We  refer the reader to \cite{Folland}, 
\cite{Charlybook} for detailed proofs and also for a discussion of numerous applications
of this space to signal analysis and quantum mechanics.

The {\em  Fock space} $\cF$ is the Hilbert space of all entire functions such that
$$
 \|F\|^2_\cF:=\int_ \bC |F(z)|^2 e^{-\pi |z|^2} dm_z <\infty,
$$
where $dm_z$ is Lebesgue measure on $\bC $.

The natural inner product in $\cF$ is  denoted by  $\lng \cdot,\cdot \rng_\cF$.

We will use the following well-known   facts:

(a)  The point evaluation is a bounded linear functional in $\cF$, 
 and  the  corresponding reproducing kernel is the function $w\mapsto
 e^{\pi \bar{z} w }$, i.e., 
\beq
\label{reproducing} 
F(z) = \lng F, e^{\pi \bar{z} \cdot } \rng_\cF \, ,  \qquad    F\in \cF.
\eeq

(b) One defines the Bargmann transform of a function $f\in L^2(\bR)$  by  
$$
  f \mapsto \cB f(z)=F(z)= 2^{1/4}e^{-\pi z^2 /2}\int_\bR f(t)e^{-\pi
    t^2}e^{2\pi t z} dt. 
$$
  
 The Bargmann transform is a unitary mapping from $L^2(\bR)$ onto $\cF$.

(c) In what follows we  identify $\bC$ and $\bR^2$. In particular
for each $\zeta=\xi +i\eta \in \bC$ we write 
$\pi_\zeta=\pi_{(\xi,\eta)}$. 
 Define the Fock space shift $\beta_\zeta: \cF \to \cF$ 
 by
$$
 \beta_\zeta F(z) = e^{i\pi \xi \eta} e^{-\pi |\zeta|^2/2}e^{\pi {\zeta} z} F(z-\bar{\zeta}).
$$
Then $\beta _\zeta $ is unitary on $\cF $,  and the Bargmann transform intertwines
the Fock space shift and the time-frequency shift: 
\beq
\label{intertwines}
\beta_\zeta \cB = \cB \pi_\zeta.
\eeq
 
(d) 
\beq
\label{ed}
\cB \vf = 2^{-1/4},
\eeq
here as above $\vf$ is the Gaussian function.

(e)  It follows from \eqref{intertwines}  and \eqref{ed}  that
$$
\cB{\pi_\zeta\vf}  = 2^{-1/4} e^{i\pi \xi \eta} e^{-\pi |\zeta|^2/2} e^{\pi \zeta z}\, .
$$
 Taking into account the reproducing property  \eqref{reproducing}, 
we can rewrite the frame property \eqref{Eq:2.8} of $\cG (\vf , \Lambda)$  as the
sampling inequality  
$$
 A\|F\|^2_\cF \leq \frac 1 {2}\sum_{\lambda\in \Lambda} |F(\overline{\lambda})|^2 e^{-\pi |\lambda|^2}
 \leq  B\|F\|^2_\cF, \qquad F\in \cF.
$$ 
In the case of square lattice, $\Lambda$ is symmetric with respect to the real line, and we have 
\beq
\label{sampling} 
 A\|F\|^2_\cF \leq \frac 1 {2}\sum_{\lambda\in \Lambda} |F(\lambda)|^2 e^{-\pi |\lambda|^2}
 \leq  B\|F\|^2_\cF, \qquad  F\in \cF.
 \eeq 
 
(f)
Let $1/2< a < 2$,  and let $w\in \bC$, $w\ne 0$. Consider the entire function
$\Phi_{a, w}(z)=e^{a\bar w z^2/w}$.  
Then
$$
|\Phi_{a,w}(z)| \asymp e^{a|z|^2}, \qquad |z-w|< 1.
$$
This statement can be checked by 
direct inspection.   
 
 \subsection{Reformulation of the main result}
 
The remaining part of Theorem \ref{th:1} can now be reformulated as follows.
 
\begin{th1}
\label{th:2}
Let $\Lambda(a)= a  \bZ^2$, and let $A(a)$ be the best possible constant in the left hand side inequality of  \eqref{sampling}. 
Then for $1/2<a<1$ we have 
$$
A(a ) \prec 1-a^2. 
$$
\end{th1}

To prove this theorem we need to find a constant $K$ and functions $F=F_a \in \cF$ 
such that
\beq
\label{example}
K(1-a^2)\|F_a\|^2\ge \sum_{\lambda\in \Lambda(a)} |F_a(\lambda)|^2 e^{-\pi |\lambda|^2}.
\eeq

\subsection{Weierstrass $\sigma$-function.} The construction of the functions $F_a$ in the
next section is motivated by the properties of the classical Weierstrass $\sigma$-function.
Let us recall its definition and basic properties. We refer the reader to \cite{Ahiezer}
for a systematic study of this function and also to \cite{superframes}
for its  applications in  Gabor analysis. 

Given a lattice $\Lambda \subset \bC$ we denote
$$
\sigma(\Lambda, z) = z \prod_{\lambda \in \Lambda \setminus \{0\}}
\Bigl(
1 - \frac z \lambda 
\Bigr)
e^{\frac z \lambda + \frac 12 \left (\frac z \lambda \right  )^2 }.
$$
This product converges uniformly on compact sets in $\bC$ to an entire function with 
$\Lambda$ as the zero set. This is a function of order $2$; moreover there exists 
$d_\Lambda \in \bC$ such that
$$
|\sigma(\Lambda, z)e^{d_\Lambda z^2}| \asymp e^{\frac \pi 2 s(\Lambda)^{-1}|z|^2}, \qquad 
\dist(\Lambda, z)\ge  \varepsilon>0.
$$ 
Here $s(\Lambda)$ is the area of the fundamental domain of $\Lambda$.
See \cite{Hayman}  and also \cite{superframes}. 

Once again, let $\Lambda(a)= a\bZ^2$.  A direct inspection shows that $d_{\Lambda(a)}=0$, so that 
\beq
\label{eq:s3}
|\sigma(\Lambda(a), z)| \asymp e^{\frac \pi 2 a^{-2}|z|^2},
\qquad \dist(\Lambda_a, z)\ge  \varepsilon>0.
 \eeq 

This relation allows one to mimic the weight function $e^{\pi
|z|^2/2}$ in the definition of the Fock space by the absolute value
of an analytic function.

 \section{Proof of   \eqref{example}}

\subsection{Explicit construction}
If $a$ is in a compact subinterval of  $(1/2,1)$,  one can take $F_a=1$
and obtain \eqref{example} with some appropriate constant $K$. 
Therefore, from now on, we assume that $a$  is sufficiently close to 1, say $0.999<a<1$. 
Given such   
$a$,  we take 
  $R=R(a)$  such that
$$
2(1-a^2 )  < R^{-3/2} < 4( 1-a^2)
$$
and  
$$
n_R:=\pi (1-R^{-3/2}) R^2 \in \mathbb N.
$$

We  need  some additional notation: 

\begin{align*}
b^2&=1-R^{-3/2}, 
\\
\zeta_{m,n}&=b^{-1}(m+in),\notag\\ 
Q_{m,n}&=\{x+iy\in \bC: |x-b^{-1}m|< b^{-1}/2,   |y-b^{-1}n|<  b\inv /2\} , \notag\\  
D_R&=\{ z \in \bC: |z| < R\} ,\notag\\ 
 D'_R&= \cup \{ Q_{m,n}: |\zeta_{m,n}| < R- 3 \} , \ D''_R = D_R \setminus D'_R, \notag\\
\cN_R&=\{ (m,n)\in \bZ ^2:   Q_{m,n} \subset D'_R\}, \notag\\  
 q_R&=\card \cN_R, \ p_R= n_R-q_R \, . \notag
\end{align*}

  We have
  \[
  \{z: R-1< |z|< R \} \subset D''_R \subset  \{z: R-4< |z|< R \}.
  \]
Using the appropriate segments of radii of the disc $D_R$ we split $D''_R$ into the "sectors" $A_k$:
 \[
D''_R=\bigcup_{k=1}^{p_R}A_k,
\]
 such that 
$$
\label{eq:20}
m< \diam A_k < M, \qquad \area A_k= b^{-2}
$$  
for some $m,M$ independent of $a$.   
Denote the center of mass of $A_k$ by 
\beq
\label{eq:22}
\zeta_k= b^2\int_{A_k}\zeta dm_\zeta.
\eeq
We can find $c$ independent of $a$ such that
$$
\{w: |w-\zeta_k|< c\} \subset A_k, \qquad k=1,\ldots, p_R.
$$

We are going to verify the estimate  \eqref{example} for the function 
\beq
\label{eq:09}
F_a(z)= z \prod_{(m,n)\in \cN_R \setminus (0,0)} \left ( 1 - \frac z {\zeta_{m,n}} \right ) 
         \prod_{k=1}^{p_R}\left (1 - \frac z{\zeta_k} \right ).
\eeq

The zero set of the function $F_a$ is 
\begin{equation}
\label{eq:911}
\cZ_a=\{\zeta_{m,n}:|\zeta_{m,n}|<R-3\}\cup\{\zeta_k\}_{k=1}^{p_R}\, .
\end{equation}
By construction, the total number of zeros of $F_a$ is $n _R = \pi R^2b^2$.   

In order to  prove \eqref{example}, we need to estimate both $\|F_a\|_\cF^2$  and 
$$
\|F_a\|^2_a:= \sum_{m,n\in \bZ}  |F_a(a(m+in))|^2 e^{-\pi a^2(m^2+n^2)}.
$$  

\subsection{Estimate of $\|F_a\|_\cF^2$. } To estimate the norm of $F_a$ in the Fock space, we compare 
the logarithm of the modulus of the polynomial $F_a$ to a subharmonic function $u_R$ whose growth 
is easy to control.   
    
  Consider the subharmonic function
$$
u_R(z) 
=    \int_{|\zeta|<R} \log \Bigl| 1- \frac z \zeta
                                                \Bigr| dm_\zeta 
                                                =     \left \{ \begin{array}{ll} 
      \frac \pi 2 |z|^2, & |z|<R,  \\
      \pi R^2\log |z| -\pi R^2\log R + \frac \pi 2 R^2,  & |z|>R .
               \end{array}  \right .
$$

 An easy estimate  shows that
$$
 u_R(z) < \frac \pi 2 |z|^2, \qquad  |z| >R.
$$

We use the following approximation lemma.
\begin{lem}
\label{l:1nn}
For each $\varepsilon >0$ there exist constants $0<c(\varepsilon)<C(\varepsilon)< \infty$ such that \begin{equation}
\label{eq:09ab}
c(\varepsilon)  |F_a(z)| < e^{  b^2 u_R(z)} < C(\varepsilon)  |F_a(z)|, \qquad   \dist(z,\cZ _a) > \varepsilon.  
\end{equation}
and
\begin{equation}
\label{eq:09ac}
c(\varepsilon)  |F_a(z)| < e^{  b^2 u_R(z)}, \qquad \dist(z,\cZ_a) \le \varepsilon.
\end{equation}
\end{lem}

 \medskip
 
 \noindent{\bf Remark.} Since the set $\cN$ is invariant with respect to rotation 
by $\pi/2$ around the origin, we find that 
$$
\sum _{(m,n)\in \cN_R \setminus (0,0)} \frac{1}{m+in} = \sum
  _{(m,n)\in \cN_R \setminus (0,0)} \frac{1}{(m+in)^2} = 0 \, .
$$  
So  the  first factor on the right-hand side of
\eqref{eq:09} in the definition of $F_a$ can be written as  
\begin{multline}
V_R(z)=z  \prod_{(m,n)\in \cN_R \setminus (0,0)} \Bigl( 1 - \frac z {\zeta_{m,n}} \Bigr) \\= 
          z\prod_{(m,n)\in \cN_R \setminus (0,0)  } \Bigl ( 1 - \frac z {\zeta_{m,n}} \Bigr) 
                  \exp \Bigl(  {\frac z {\zeta_{m,n}}+ \frac 12 \Bigl(\frac z {\zeta_{m,n}}\Bigr)^2} 
                  \Bigr).
                  \label{n68}
\end{multline}              
Consequently, the function $V_R$ can be viewed as  a truncated  version of the  Weierstrass
 $\sigma$-function and estimates \eqref{eq:09ab}  and \eqref{eq:09ac}
 correspond to the 
 growth estimate \eqref{eq:s3} for  the  Weierstrass
 $\sigma$-function.

We postpone the proof of this
technical lemma until subsection \ref{last}. 
Assuming that Lemma \ref{l:1nn} is already proved, an estimate of  $\|F_a\|_\cF^2$  is  straightforward. 
 
\begin{lem}
\label{l:02n} 
$$
 \|F_a\|^2_\cF \succ R^{3/2}  \asymp (1-a^2)^{-1}.
$$
\end{lem}

\begin{proof}
Let $\Omega=\{z\in\mathbb C:|z|<R,\,\dist(z,\mathcal Z_a)>1/10\}.
$ Using Lemma~\ref{l:1nn}, we find that 
$$
\|F_a\|^2_\cF \succ  \int _\Omega
e^{2 b^{2}u_R(z)-\pi |z|^2} dm_z =    I(a,R).
$$
We use that $1-b^2= R^{-3/2}$. Furthermore, for every $1<r<R$, the circle $z:|z|=r$ intersects with $\Omega$
on at most half of its length.
Therefore,   
$$
I(a,R) \succ \int_1^{R}e^{-\pi R^{-3/2}t^2}  t dt=
     \frac{R^{3/2}}2\int_{R^{-3/2}}^{R^{1/2}}e^{-\pi u}  du \asymp R^{3/2},
$$ 
and the statement of the Lemma now follows.
\end{proof}
  
\noindent{\bf Remark. } {\em  A similar argument shows that } 
\beq
\label{eq:24}
\int_{|z|>R-4}|F_a(z)|^2 e^{-\pi |z|^2}  dm_z \to 0, 
\eeq
{\em as $a\to1$ or equivalently, as $R\to \infty$.}

\subsection{Estimate of $\|F_a\|_a^2$. } 
\begin{lem}
\label{l:03}
For $F_a$ as in \eqref{eq:09} we have 
\begin{equation}
\label{eq:09c}
\|F_a \|^2_a= \sum_{m,n} |F_a(a(m+in))|^2 e^{-\pi a^2(m^2+n^2)}\asymp 1.
\end{equation}
\end{lem}
 \begin{proof}
We have
$$ 
 \|F_a \|^2_a= 
 \Bigl(
 \sum_{(m,n)\in \cN_R} + \sum_{(m,n)\not\in \cN_R} 
 \Bigr) |F_a(a(m+in))|^2 e^{-\pi a^2(m^2+n^2)}
 = \Sigma_1(a) + \Sigma_2(a).
$$

In order to estimate $ \Sigma_1(a)$,   we observe that 
$F_a(\zeta_{m,n})=0$, for $(m,n)\in \cN_R$ and  
$$
|\zeta_{m,n}- a(m+in)| = (b^{-1}-a)|m+in| < 2R^{-3/2}|m+in|.
$$
Since $|m+in|< R$ for $(m,n)\in \cN_R$, and 
$a$ is sufficiently close to $1$, we have
$$
|\zeta_{m,n}- a(m+in)| < \frac 15, \qquad (m,n)\in \cN_R. 
$$

Denote $D_{m,n}=\{z\in \bC: |z -\zeta_{m,n}| < 1/4\}$. By
part (f) of subsection~\ref{FM} there exists a function  
$\Phi_{m,n}(z)$ that is holomorphic on $D_{m,n}$ and satisfies
\beq
\label{eq:s4a}
|\Phi_{m,n}(z)| \asymp e^{b^2\frac \pi 2 |z|^2}, \qquad  z\in D_{m,n}.
\eeq

Then for  each $(m,n)\in \cN_R$ the function
$$
\Psi_{m,n}(z) = \frac {F_a(z)}{\Phi_{m,n}(z)} 
$$
is  holomorphic in $D_{m,n}$ and possesses the properties 
$$
   \Psi _{m,n}(\zeta_{m,n})= 0 \ \mbox{and}  \ |\Psi _{m,n}(z)|\prec 1.
$$  
   
By Cauchy's theorem, the functions $\Psi'_{m,n}$ are
uniformly bounded on $D^*_{m,n}=\{z\in \bC: |z -\zeta_{m,n}| < 1/5\}$, and hence
$$
|\Psi_{m,n}(a(m+in))| \prec   |(a-b^{-1})(m+in)| \prec  R^{-3/2}(m^2+n^2)^{1/2}, \quad (m,n)\in \cN_R.
$$ 
Returning to the  function $F_a$ and using
\eqref{eq:s4a}   once again, we obtain
$$
|F_a(a(m+in))|^2e^{-\pi a^2 (m^2+n^2)}  \prec 
R^{-3}(m^2+n^2) |\Phi_{m,n}(a(m+in))|^{2(1-b^{-2})}.
$$
The mean value inequality for $\bigl|\Phi_{m,n}(a(m+in))^{2(1-b^{-2})}\bigr|$ now yields
\begin{multline*}
|F_a(a(m+in))|^2e^{-\pi a^2 (m^2+n^2)} \prec 
 R^{-3}(m^2+n^2) 
    \int_{D_{m,n}} |\Phi_{m,n}(z)|^{2(1-b^{-2})} dm_z \\ \prec 
    R^{-3} \int_{D_{n,m}}  |z^2| e^{\pi (b^2-1)|z|^2}dm_z.
\end{multline*}
Since all discs $D_{m,n}$  are disjoint we obtain 
\beq
\label{eq:32} 
\Sigma_1(a) \prec  R^{-3}\int_{|z|<R}|z^2| e^{\pi (b^2-1)|z|^2}dm_z 
\prec    R^{-3}\int_0^R t^2 e^{-\pi R^{-3/2}t^2} tdt \prec 1.
\eeq

\medskip

Finally, for arbitrary $(m,n)$, the mean value theorem yields
 \[
|F_a(a(m+in))|^2e^{-\pi a^2 (m^2+n^2)} \prec  \\
      \int_{D_{m,n}} |F_a(z)|^2 e^{-\pi |z|^2} dm_z,
 \]
and, by \eqref{eq:24} we obtain
\beq
\label{eq:a14}
\Sigma_2(a)  \prec \int_{|z|>R-4} |F_a(z)|^2 e^{-\pi |z|^2} dm_z \to  0, \ 
\mbox{as}  \  a \nearrow 1. 
\eeq

The estimates \eqref{eq:32} and \eqref{eq:a14} yield
\begin{equation}
\|F_a \|^2_a= \sum_{m,n} |F_a(a(m+in))|^2 e^{-\pi a^2(m^2+n^2)}\prec 1.
\label{e681}
\end{equation}
The opposite relation follows from Lemma~\ref{l:02n} and from the lower estimate on $A(a)$ established in 
Proposition~\ref{lower}. 
\end{proof}
    
\medskip
 
Relation \eqref{example} follows immediately from Lemma~\ref{l:02n} and the estimate \eqref{eq:09c} (or even \eqref{e681}).
 
\medskip
 

\subsection{Proof of the approximation lemma}
\label{last}
The proof of Lemma \ref{l:1nn} is based on atomization techniques, see e.g.
\cite{atomization}. First we rewrite \eqref{eq:09ab} and
\eqref{eq:09ac}  in an additive form. We must prove that 
 
\begin{equation}
\label{eq:09a}
\log |F_a(z)|  = b^2 u_R(z)   + O(1), \qquad \dist(z,\mathcal{Z}_a) > \varepsilon,
\end{equation}
and 
$$
\log |F_a(z)|  \leq   b^2 u_R(z) +O(1), \qquad \dist(z,\cZ _a) \le  \varepsilon,
$$
where $\cZ _a=\{\zeta_{m,n}: (m,n)\in \cN_R\}\cup
\{\zeta_k\}_{k=1}^{p_R}$ as in \eqref{eq:911},
and the quantities $O(1)$  in the right-hand sides of these relations are
 bounded uniformly with respect to all $a\in (0.999,1)$ and depend only
 on $\varepsilon$.  It suffices to prove \eqref{eq:09a}, the second
 relation will then follow by the maximum principle applied to $F_a\Phi_{a,w}^{-1}$ where 
$\Phi_{a,w}$ is defined in part (f) of subsection~\ref{FM}.  
 
Let $V_R$ be defined by \eqref{n68},
$$
   v_R(z)= b^{2} \int_{D'_R} \log\left  |1 - \frac z{\zeta} \right |  dm_\zeta,
$$ 
 and let
$$
 w_R(z)= b^{2} \int_{D''_R} \log \left |1 - \frac z{\zeta} \right |dm_\zeta,
\qquad       W_R(z)= \prod_1^{p_R}\left (1- \frac z {\zeta_k} \right ).
$$
  
 We have 
\begin{multline} 
\label{eq:b06}
  \log |F_a(z)| -  b^2 u_R(z) = \\ \left (\log |V_R(z)| -    v_R(z) \right ) +
       \left (\log |W_R(z)| -    w_R(z) \right )   = \mathfrak S_1(R,z)+ \mathfrak S_2(R,z),
 \end{multline}
and we estimate separately each summand in the right-hand side of \eqref{eq:b06}.

Let 
 $ \dist(z,\cZ _a) > \varepsilon$.  We have 
\begin{multline*}
 \mathfrak S_1(R,z) =  \log |V_R(z)| - v_R(z)=  
b^2\int_{Q_{0,0}} \left ( 
                                  \log  |z| -
                                             \log\left |1 - \frac z \zeta\right |  
             \right ) dm_\zeta
 \\ 
+  b^2\Bigl(
     \sum_{\substack{
                  (m,n)\in \cN_R\setminus\{0,0\}, \\ \mbox{dist}(z,Q_{m,n})\leq 10
                 }} +
      \sum_{\substack{(m,n)\in \cN_R\setminus\{0,0\}, \\ \mbox{dist}(z,Q_{m,n})> 10
                }  }    
    \Bigr)  
            \underbrace{
             \int_{Q_{m,n}} \left ( 
                                  \log \left |1-\frac z{\zeta_{m,n} } \right | -
                                             \log\left |1 - \frac z \zeta\right |  
             \right ) dm_\zeta
                         }
                            _{j_{m,n}}.
 \end{multline*}   

 It suffices to estimate just the second sum  in the right-hand side because the first sum
contains only a finite number (at most $1000$, say) of uniformly bounded terms, and the first
integral is bounded uniformly in $z$. 

 Denote  $L(\zeta)=\log(1 - z/\zeta)$. We then have 
$$
 j_{m,n}= \Re \Bigl [ \int_{Q_{m,n}}\bigl(
L(\zeta)-L(\zeta_{m,n})dm_\zeta                                            
                                            \bigr )\Bigr].
$$
We apply the second order Taylor expansion with the remainder term in the  integral form:
\[
L(\zeta)-L(\zeta_{m,n}) = L'(\zeta_{m,n})(\zeta-\zeta_{m,n}) + 
    \frac 12 L''(\zeta_{m,n})(\zeta-\zeta_{m,n})^2 + 
         \frac 1 2 \int_{\zeta_{m,n}}^\zeta L'''(s) (\zeta-s)^2 ds.
 \]        
 and use the fact that
\[
\int_{Q_{m,n}} (\zeta-\zeta_{m,n})dm_\zeta=
   \int_{Q_{m,n}} (\zeta-\zeta_{m,n})^2dm_\zeta=0.
\] 
Then 
\begin{multline*}
|j_{m,n}|= \Bigl | 
      \int_{Q_{m,n}} \int_{\zeta_{m,n}}^\zeta 
      {(\zeta-s)^2} \left ( \frac 1 {(s-z)^3} -  \frac 1{s^3} \right )  ds\, dm_\zeta
\Bigr|  
  \\  \prec \frac 1 {\mbox{dist}(z,Q_{m,n})^3}+ \frac 1  {\mbox{dist}(0,Q_{m,n})^3},
\end{multline*}
which implies that 
$$
\mathfrak S_1(R,z)  = O(1).
$$

\medskip

Finally, 
\begin{multline*}
\mathfrak S_2(R,z)= | W_R(z) - w_R(z) |=    \\ \Bigl | b^2\Bigl (  
 \sum_{ \mbox{dist}( z, A_k) \le M+10} +\sum_{\mbox{dist}( z, A_k) >M+10} 
                                \Bigr)
\underbrace{
\int_{A_k} \left (\log \left |1 - \frac z{\zeta_k} \right |-
             \log \left |1- \frac z{\zeta} \right | \right )dm_\zeta
                    }_{i_k} 
                                  \Bigr|,
\end{multline*}                    
with $M$ as in \eqref{eq:20}.                                
  The first term in the right-hand side contains just a finite number of summands and 
 is always bounded. 
 In order to estimate each $i_k$ from the second term we use the Taylor formula (now of the 
 first order) with the same function $L(\zeta)=\log (1-  z / \zeta)$.
The choice of $\zeta _k$ in \eqref{eq:22} implies that 
 $\int _{A_k} (\zeta -\zeta _k) dm_\zeta    = 0$. Arguing as above, 
 we  obtain
  \[
 |i_k| \prec \frac 1 {\dist (z, A_k)^2} + \frac 1  {\dist (0, A_k)^2},  
 \]
 whence  
$$
\mathfrak S_2(R,z)  = O(1).
$$
This completes the proof of Lemma \ref{l:1nn}.  $\Box$

\end{document}